\documentclass{amsart}

\usepackage[utf8]{inputenc}
\usepackage{amsmath,amssymb,amsthm}
\usepackage{tikz-cd}
\usepackage{mathtools}
\usepackage{enumitem}
\usepackage{bbold}
\usepackage{hyperref}

\newtheorem{theorem}{Theorem}
\newtheorem{lemma}{Lemma}[subsection]
\newtheorem{definition}{Definition}[section]
\newtheorem{proposition}{Proposition}[section]
\newtheorem*{corollary}{Corollary}
\newtheorem{example}{Example}[section]

\newtheorem{question}{Question}
\newtheorem*{notation}{Notation}

\newtheorem{remark}{Remark}[section]

\newcommand{\GZip}{\mathop{\text{$G$-{\tt Zip}}}\nolimits}

\newcommand{\GZipmu}{\mathop{\GZip^{\mu}}\nolimits}

\newcommand{\A}{CH^{\bullet}}

\newcommand{\Kz}{K_{0}}
\newcommand{\Kn}{K_{n}}
\newcommand{\R}{R}
\newcommand{\rkWeyl}{|W|}
\newcommand{\rkLWeyl}{|W_{L}|}
\newcommand{\Kphi}{K_{0,\varphi}}
\newcommand{\KTG}{\Kphi^{T}(G)}
\newcommand{\KLG}{\Kphi^{L}(G)}
\newcommand{\KTnG}{K_{n,\varphi}^{T}(G)}

\newcommand{\der}{\rm der}

\newcommand{\SL}{\mathsf{SL}}

\newcommand{\Diag}{\mathsf{Diag}}

\DeclareMathOperator{\Cent}{Cent}

\DeclareMathOperator{\End}{End}
\DeclareMathOperator{\Vect}{Vect}
\DeclareMathOperator{\Lie}{Lie}
\DeclareMathOperator{\Rep}{Rep}

\title{Grothendieck group of the stack of G-Zips}
\author{Simon Cooper}

\begin{document}

\begin{abstract}
    Given a connected reductive group G over the finite field of order p and a cocharacter of G over the algebraic closure of the finite field, we can define G-Zips. The collection of these G-Zips form an algebraic stack which is a stack quotient of G. In this paper we study the K-theory rings of this quotient stack, focusing on the Grothendieck group. Under the additional assumption that the derived group is simply connected, the Grothendieck group is described as a quotient of the representation ring of the Levi subgroup centralising the cocharacter.
\end{abstract}
\maketitle
\section{Introduction}
	Throughout the paper, unless stated otherwise, $p>0$ is a prime and $k$ an algebraic closure of $\mathbb{F}_{p}$.
	The study of moduli spaces in positive characteristic has been enriched with the introduction of the algebraic stack $\GZipmu$ in \cite{PWZ} building on the work of \cite{MoonenWedhorn}. The input is a \textit{cocharacter datum}, which is defined to be a pair $(G,\mu)$, where $G$ is a connected reductive group over $\mathbb{F}_{p}$ and $\mu \colon \mathbb{G}_{m} \rightarrow G_{k}$ is a cocharacter. Given a cocharacter datum $(G,\mu)$, Pink-Wedhorn-Ziegler define a $G$-Zip of type $\mu$ over a base $k$-scheme $S$ \cite{PWZ}. The category of $G$-Zips of type $\mu$ forms an algebraic stack denoted $\GZipmu$. We are interested in stacks $X$ admitting a morphism $X \rightarrow \GZipmu$. Following \cite{CooperGoldring} and \cite{Goldring-Koskivirta-rank-2-cones} such a morphism is called a \textit{Zip period map}. The typical example is for $X$ the special fibre of a Shimura variety of Hodge type, see \cite{MoonenWedhorn}. One vague general question about Zip period maps, initially raised as Question A in \cite{GoldringKoskivirta2}, is the following.
	\begin{question}
		\label{q-geom-zip}
		Let $X$ be a stack and $X \rightarrow \GZipmu$ a Zip period map. What, if anything, does the geometry of $\GZipmu$ tell us about the geometry of $X$?
	\end{question}
	Clearly a fundamental part of an answer lies in understanding the geometry of $\GZipmu$. Indeed, this is the trivial case $X=\GZipmu$ of Question \ref{q-geom-zip}. In this paper we hope to contribute towards answering this by studying one specific invariant: the $\Kz$-ring.
	\begin{question}
		\label{q-kzero}
		Let $(G,\mu)$ be a cocharacter datum. What is $\Kz(\GZipmu)$?
	\end{question}
    $\GZipmu$ admits a presentation as a quotient stack $\GZipmu \simeq [G_{k}/E]$ (\ref{prop-gzip-presentation}) and for the sake of this paper we define $\Kz(\GZipmu) \coloneqq \Kz^{E}(G_{k})$ the $\Kz$-ring of the exact category $\Vect^{E}(G_{k})$ of $E$-equivariant vector bundles on $G_{k}$. Note that this should coincide with $\Kz$ of the stack $\GZipmu$ for any reasonable definition of $K$-theory of stacks.
	The motivation for focusing on this invariant in particular is that it is closely related to the Chow ring with rational coefficients $\A(\GZipmu)$, which is known by work of Brokemper \cite{Brokemper}. In this paper we provide a partial answer to Question \ref{q-kzero}: we compute $\Kz(\GZipmu)$ in the case where $G^{\der}$ is simply connected. \begin{theorem}
		\label{th-main}
		Let $(G,\mu)$ be a cocharacter datum. Suppose that $G^{\der}$ is simply connected. 
		\[\Kz(\GZipmu) \cong \R(L)/I\R(L)\]
		where $I \coloneqq \{c-\varphi(c) \mid c \in \R(G)\} \subset \R(G)$, with $\varphi$ denoting the map induced from Frobenius and $L \coloneqq Cent_{G}(\mu)$.
	\end{theorem}
	\begin{remark}
		This was known in the rational $\Kz$-case using methods from motivic cohomology \cite[5.4]{Yaylali}.
	\end{remark}
	There are some technical barriers to adapting this to the general case related to the fact that the representation theory is more complicated when $G^{\der}$ is not simply connected. In future work we hope the results of this paper can be extended in three directions: tackling the general case (when $G^{\der}$ not simply connected), saying something about the higher $K$-groups and studying Question \ref{q-geom-zip} in the context of $K$-theory.
	
	\subsection{Sketch proof of Theorem \ref{th-main}}
	\label{sec-intro-proof}
	   Firstly, $\Kz^{E}(G_{k}) \simeq \Kz^{L}(G_{k})$ where $L \coloneqq \Cent_{G}(\mu)$ since equivariant $K$-theory ignores the ‘unipotent part of the action’ and $L$ is the reductive quotient of $E$ (\ref{th-ignore-unip}).
    \begin{notation}
        Denote $\KLG \coloneqq \Kz^{L}(G_{k})$ to emphasise that this is not the standard action of $L$ on $G_{k}$, but rather an action by $\varphi$-conjugation $l\cdot g = lg\varphi(l)^{-1}$.
    \end{notation}
  For sake of exposition we often write $G$ instead of $G_{k}$; all schemes are considered over $k$.
	
	\vspace{3mm}
	
	The embedding $T \subset L$ of the maximal torus into the Levi subgroup induces a restriction map $r^{L}_{T}\colon \KLG \hookrightarrow \KTG$, which is injective. To compute $\KTG$ we ‘untwist’ the torus action. The embedding $T \hookrightarrow T\times T$ given by $t\mapsto(t,\varphi(t))$ induces a surjective restriction map $\Kz^{T\times T}(G) \rightarrow \KTG$. The equivariant Kunneth formula (\ref{prop-equiv-kunneth-form}) gives 
    $$\Kz^{T\times T}(G) \simeq \Kz^{T}(G/B) \simeq \Kz^{T}(*)\otimes_{\R(G)}\R(T)$$ The resulting map $\R(T) \cong \Kz^{T}(*) \rightarrow \KTG$ is surjective with kernel $I\R(T)$ (\ref{cor-kthy-tor-act}).
	
	\vspace{3mm}
	
	Given that $G^{\der}$ is simply connected we can use the equivariant Kunneth formula $\R(T)\otimes_{\R(L)}\KLG \simeq \KTG$ together with the fact that $\R(L)\hookrightarrow \R(T)$ is faithfully flat to deduce $\KLG$ from $\KTG$ (\ref{th-main-intext}).

	\subsection{Outline of the paper}
	The stack $\GZipmu$ is defined in Section \ref{sec-Gzips}. Background from equivariant algebraic $K$-theory is presented in Section \ref{sec-eq-kthy}. This consists of classical results due to Thomason \cite{Thomason} alongside results in \cite{Merkurjev} building on Thomason's work. The results regarding representation rings often predate this.  Section \ref{sec-untwist} computes the ring $\KTG$. This is used to deduce the ring $\KLG$ in Section \ref{sec-main-th-proof}.  Section \ref{sec-altproof-comp-invariants} presents an alternative approach to the results in Section \ref{sec-main-th-proof} using the description of the equivariant $K$-theory as the ring of Hecke invariants on the $K$-theory of the torus action, due to \cite{MegumiLandweberSjamaar}. 

 	\section*{Acknowledgements}
	
	We thank A. Merkurjev, R. Sjamaar, V. Uma and C. Yaylali for helpful conversations during the writing of this paper. Special thanks to M. Brion for a careful reading of earlier drafts.

	\section{The stack of \texorpdfstring{$G$}{G}-Zips of type \texorpdfstring{$\mu$}{mu}}
	\label{sec-Gzips}
	Recall that $p>0$ is a prime and $k$ is an algebraic closure of $\mathbb{F}_{p}$.
	\begin{notation}
		Denote $\sigma \colon k \rightarrow k$ the absolute Frobenius morphism and for any $k$-scheme $X$ write $X^{(p)} \coloneqq X \times_{\sigma}k$. Denote $\varphi \colon X \rightarrow X^{(p)}$ the relative Frobenius.
	\end{notation}
	\begin{definition}
		A cocharacter datum is a pair $(G,\mu)$ where $G$ is a reductive group over $\mathbb{F}_{p}$ and $\mu \colon \mathbb{G}_{m,k} \longrightarrow G_{k}$ is a cocharacter.
	\end{definition}  
	This datum determines parabolic subgroups $P^{-},P^{+} \subset G_{k}$, such that the Lie algebra of the parabolic $P^{-}$ (resp. $P^{+}$) is the sum of the non-positive (resp. non-negative) weight spaces for the action $\mu$ on $\Lie(G)$. Let $P \coloneqq P^{-}$ and $Q \coloneqq (P^{+})^{(p)}$ and  $L \coloneqq P^{-}\cap P^{+}$ and $M \coloneqq L^{(p)}$  be their Levi factors. 
	\begin{definition}[\cite{PWZ} 1.4, 3.2]
		Let $(G,\mu)$ be a cocharacter datum, $P$, $Q$ be as above and $S$ be a $k$-scheme. A $G$-zip of type $\mu$ over $S$ is a tuple $(I,I_{P},I_{Q},\iota)$ where $I$ is a right $G_{k}$-torsor over S, $I_{P}\subset I$ is a right $P$-torsor over $S$, $I_{Q}\subset I$ is a right $Q$-torsor over $S$ and $$\iota \colon I_{P}^{(p)}/\mathcal{R}_{u}(P)^{(p)} \longrightarrow I_{Q}/\mathcal{R}_{u}(Q)$$ is an isomorphism of $M$-torsors over $S$.
		The category of $G$-zips of type $\mu$ over $k$-schemes forms an algebraic stack $\GZipmu$.
	\end{definition}   Here $\mathcal{R}_{u}(P)$, $\mathcal{R}_{u}(Q)$ denote the unipotent radicals of $P$, $Q$.
	 There is a presentation of $\GZipmu$ as a quotient stack as follows. Write $\pi_{P}\colon P \rightarrow L$ and $\pi_{Q}\colon Q \rightarrow M$ for the projections to the Levi factors. Define the zip group $$E \coloneqq \{(x,y) \in P \times Q \mid \varphi(\pi_{P}(x)) = \pi_{Q}(y)\}\subset G_{k}\times G_{k}$$  The zip group sits in a split short exact sequence 
	\begin{equation}
	\label{tikzcd-ses-zipgroup}
	\begin{tikzcd}
	0 \arrow[r]& \mathcal{R}_{u}(P)\times\mathcal{R}_{u}(Q) \arrow[r]& E \arrow[r]& L\arrow[r] & 0
	\end{tikzcd}
	\end{equation}
	where the splitting $L \hookrightarrow E$ is given by $l \mapsto (l,\varphi(l))$. There is an action of $E$ on $G_{k}$ given by $(x,y)\cdot g = xgy^{-1}$.
	\begin{proposition}[\cite{PWZ} 3.11]
		\label{prop-gzip-presentation} There is an isomorphism
	 $\GZipmu \simeq [E\backslash G_{k}]$.
	\end{proposition} This implies that $\GZipmu$ is a smooth algebraic stack over $k$ of dimension $0$. The aim of the paper is to compute $\Kz(\GZipmu) = \Kz^{E}(G_{k})$.
	
 \section{Simply-connectedness assumption}
	Let $G$ be a reductive group over an algebraically closed field $k$. Given a maximal torus $T \subset G$ we obtain a root datum $(X^{*}(T),\Phi,X_{*}(T),\Phi^{\vee})$. A choice of Borel subgroup $T \subset B \subset G$ is equivalent to a choice of simple roots $\Delta \subset \Phi$. The fundamental weights $\eta_{\alpha} \in X^{*}(T)_{\mathbb{Q}}$ are defined to be dual to the simple coroots $\alpha^{\vee} \in \Phi^{\vee}$ under the natural pairing $X^{*}(T)_{\mathbb{Q}}\otimes X_{*}(T)_{\mathbb{Q}} \rightarrow \mathbb{Q}$. That is, $(\eta_{\alpha},\beta^{\vee}) = \delta_{\alpha,\beta}$.
    \begin{definition}
         We say a connected semisimple group $G$ is simply connected if for some (equivalently every) maximal torus $T\subset G$ and choice of simple roots $\Delta \subset \Phi$, the fundamental weights lie in the character group lattice: $\{\eta_{\alpha}\} \subset X^{*}(T)$.
    \end{definition}
    \begin{lemma}
    \label{le-simply-connected-Levi}
        Let $G$ be a simply connected semisimple group over an algebraically closed field and $L\subset G$ a Levi subgroup. Then $L^{\der}$ is simply connected.
    \end{lemma}
    \begin{proof}
         This is a standard result. A proof can be found in \cite[12.14]{Malle-Testerman} 
    \end{proof}
    \begin{corollary}
    \label{cor-simply-conn-Levi}
        Let $G$ be a reductive group over an algebraically closed field and $L \subset G$ a Levi subgroup. If $G^{\der}$ is simply connected then $L^{\der}$ is simply connected.
    \end{corollary}
    \begin{proof}
        $L\cap G^{\der} \subset G^{\der}$ is a Levi subgroup and $(L\cap G^{\der})^{\der} = L^{\der}$ is simply connected by Lemma \ref{le-simply-connected-Levi}.
    \end{proof}
 
 \section{General results from equivariant algebraic \texorpdfstring{$K$}{}-theory}
	\label{sec-eq-kthy}
	This section recalls some general results from equivariant $K$-theory which are required for the remainder of the paper.
	\begin{notation}
		By a $G$-variety we mean a smooth $k$-variety $X$ with an algebraic group action of $G$ on $X$. Given $X$ a $G$-variety, denote $\Kn^{G}(X) \coloneqq \Kn(\Vect^{G}(X))$ the $n$-th $K$-theory group of the category of $G$-vector bundles on $X$ as in \cite{Thomason}. Given the trivial $G$-action on the point we obtain an equivalence of categories $\Vect^{G}(*) \simeq \Rep(G)$, inducing an isomorphism 
    $\Kz^{G}(*) \simeq \R(G)\coloneqq \Kz(\Rep(G))$. There is an $\R(G)$-module structure on $\Kn^{G}(X)$. If $H \rightarrow G$ is a homomorphism of algebraic groups then we write $r^{G}_{H}\colon \Kn^{G}(X)\rightarrow \Kn^{H}(X)$ for the restriction map. It is compatible with the restriction $\R(G)\rightarrow \R(H)$. There is an $\R(G)$-algbera structure on $\Kz^{G}(X)$, with multiplication given by tensor product of vector bundles. Given the rank homomorphism $\epsilon\colon \R(G) \rightarrow \mathbb{Z}$, the augmentation ideal is defined by $J(G) \coloneqq \ker(\epsilon)$. 
	\end{notation}
    
	\subsection{Ignore unipotent action}
	The following result says that the unipotent part of the group action plays no role in the equivariant $K$-theory ring. 
	\begin{theorem}
		\label{th-ignore-unip}
		Let $k$ be an algebraically closed field, $G$ an arbitrary linear algebraic $k$-group (not connected reductive as in the rest of the paper) and $X$ be a $G$-variety. Let \[\begin{tikzcd}
		1 \arrow[r]& K \arrow[r]& G \arrow[r]& H \arrow[r]& 1
		\end{tikzcd}\] be a split short exact sequence of $k$-groups with $H$ reductive and $K$ unipotent. For all $n\geq0$ the restriction map $r^{G}_{H}\colon\Kn^{G}(X)\rightarrow \Kn^{H}(X)$ is an isomorphism. For $n=0$ this is an isomorphism of rings.
	\end{theorem}
	\begin{proof}
		The theorem follows from foundational results of Thomason \cite{Thomason}. A proof appears in \cite[5.2.18]{ChrissGinzburg}.  
	\end{proof}
	
	\subsection{Representation rings}
	\label{subsec-weyl-action}
	Let $G$ be a connected reductive group and $T \subset G$ a maximal torus (here we can work over an arbitrary algebraically closed field). We recall some classical results regarding representation rings and Weyl group actions. These will allow us to reduce to the torus action. The Weyl group is $W = W(G,T) \coloneqq N_{G}(T)/T$. There is an action of $W$ on $X^{*}(T)$ by $(nT\cdot \chi)(s) = \chi(nsn^{-1})$. This extends to an action of $W$ on $\R(T)$. 
	
	\begin{lemma}[\cite{Tits}]
		\label{le-weyl-action-repring}
		The restriction map $\R(G)\hookrightarrow \R(T)$ induces an isomorphism $\R(G) \simeq \R(T)^{W}$. 
	\end{lemma}
	Under the additional assumption that $G^{\der}$ is simply connected there exists a basis for $\R(T)$ over the subring $\R(G)$. First, we note that this assumption is equivalent to $\pi_{1}(G)$ being torsion-free.
	\begin{lemma}[\cite{Merkurjev} 1.7]
		Let $G$ be a reductive group. The derived group $G^{\der}$ is simply connected if and only if $\pi_{1}(G)$ is torsion-free.
	\end{lemma}
	\begin{lemma}[\cite{Steinberg} 1.3]
		\label{le-steinberg-pittie}
		Suppose that $G^{\der}$ is simply connected. The ring $\R(T)$ is a free $\R(G)$-module of rank $\rkWeyl$. 
	\end{lemma}
	\begin{corollary}
		\label{cor-repring-fflat}
		Suppose that $G^{\der}$ is simply connected. $\R(G) \hookrightarrow \R(T)$ is faithfully flat. 
	\end{corollary}
	
	\subsection{Equivariant Kunneth formula}
	The fundamental result relating the equivariant algebraic $K$-theory of a reductive group $G$ with that of its maximal torus $T$ is the equivariant Kunneth formula due to Merkurjev. This holds over an arbitrary algebraically closed field.
 
	\begin{proposition}[Equivariant Kunneth formula \cite{Merkurjev} 4.1]
		\label{prop-equiv-kunneth-form}
		Let $G$ be a reductive group, $T \subset G$ a maximal torus and $X$ a $G$-variety. Suppose that $G^{\der}$ is simply connected. For all $n\geq0$ the restriction map $r^{G}_{T}\colon \Kn^{G}(X)\rightarrow \Kn^{T}(X)$ factors through an isomorphism $\Kn^{G}(X)\otimes_{\R(G)}\R(T) \xrightarrow{\sim} \Kn^{T}(X)$. For $n=0$ this is an isomorphism of rings.
	\end{proposition}
		
	\begin{remark}
		It does not follow that $\Kz^{G}(X)\simeq \Kz^{T}(X)^{W}$, see Example \ref{example-weyl-action}.
	\end{remark}	

 	\subsection{Changing the group}
	These technical results will be used for computing $\KTG$ in the following section.
    \begin{proposition}[\cite{Thomason} 6.2]
	\label{prop-Thoma-subgp}
		Let $X$ be a $G$-scheme and $H\subset G$ a closed subgroup. Restriction along $X\times \{H\} \hookrightarrow X \times G/H$ induces an isomorphism $\Kn^{G}(X\times G/H) \xrightarrow{\simeq} \Kn^{H}(X)$ for all $n\geq 0$. For $n=0$ this is an isomorphism of rings.
	\end{proposition}
 	\begin{proposition}
 	    Let $X$ be a $G$-scheme, $\chi \in X^{*}(G)$ and $H \coloneqq \ker(\chi)$. There is a long exact sequence
      \[
      \begin{tikzcd}[column sep=small]
          \ldots \arrow[r] & K_{n+1}^{H}(X) \arrow[r] & \Kn^{G}(X)\arrow[r, "1-\chi"] & \Kn^{G}(X) \arrow[r, "r_{H}^{G}"] & \Kn^{H}(X) \arrow[r] & K_{n-1}^{G}(X) \arrow[r] & \ldots
      \end{tikzcd}
      \]
 	\end{proposition}
  \begin{proof}
      This follows from the localisation long exact sequence of equivariant $K$-theory.
  \end{proof}
  \begin{corollary}[\cite{Merkurjev} 2.13]
		\label{prop-Merk-char}
		Let $X$ be a $G$-variety, $\chi \in X^{*}(G)$ and $H \coloneqq \ker(\chi)$. 
		\[\begin{tikzcd}
		\Kz^{G}(X) \arrow[r, "1-\chi"]& \Kz^{G}(X) \arrow[r]& \Kz^{H}(X) \arrow[r]& 0
		\end{tikzcd}\] is exact. In particular, $\Kz^{H}(X) \cong \Kz^{G}(X)/(1-\chi)\Kz^{G}(X)$ as rings.
	\end{corollary}
	\begin{corollary}
		\label{cor-subtor-quot-ktheory}
		Let $T$ be a torus, $S \subset T$ a sub-torus and $X$ a $T$-variety. The restriction map $\Kz^{T}(X)\rightarrow \Kz^{S}(X)$ is surjective with kernel $$\langle1-\chi \mid \chi \in X^{*}(T)\text{ such that } \chi|_{S} = 1\rangle\Kz^{T}(X)$$
	\end{corollary}
	
	\begin{proof}
		One can choose $\chi_{1},\ldots,\chi_{n} \in X^{*}(T)$ with $\chi_{i}|_{S} = 1$ such that defining $T_{1} \coloneqq \ker(\chi_{1})$, $T_{2} \coloneqq \ker(\chi_{2}|_{T_{1}})$, \ldots  , $T_{n} \coloneqq \ker(\chi_{n}|_{T_{n-1}})$ one has $T_{n} = S$ and the result follows from successive application of Proposition \ref{prop-Merk-char}.
	\end{proof}

 \section{The torus action}
	\label{sec-untwist}
	Recall that $p>0$ is a prime and $k$ is an algebraic closure of $\mathbb{F}_{p}$. 
    Throughout Sections \ref{sec-untwist} and \ref{sec-main-th-proof}, let $G$ be a reductive group over $k$ such that $G^{\der}$ is simply connected. Let $T \subset G$ be a maximal torus, $L \subset G$ a Levi subgroup containing $T$ and $\varphi \colon G \rightarrow G$ the Frobenius homomorphism.
	
	\vspace{3mm}
	
	The aim of this section is to compute $\KTG$, the equivariant algebraic $\Kz$-ring of the $T$-action on $G$ given by $t\cdot g \coloneqq tg\varphi(t)^{-1}$. This requires ‘untwisting’ the $T$-action, which can be done by employing a trick used by Brokemper \cite{Brokemper} and considering the short exact sequence 
	\[\begin{tikzcd}
	1 \arrow[r]& T \arrow[r, "\iota"]& T\times T \arrow[r, "p"]& T \arrow[r] &1
	\end{tikzcd}\] where $\iota(t) = (t,\varphi(t))$ and $p(t_{1},t_{2}) = \varphi(t_{1})t_{2}^{-1}$. Taking the conjugation action $((t_{1},t_{2}),g) \mapsto t_{1}gt_{2}^{-1}$ of $T\times T$ on $G$ gives a pullback map $r^{T\times T}_{T}\colon \Kz^{T\times T}(G)\rightarrow\KTG$ to the desired ring. 
	\begin{proposition}	\label{prop-torus-action-quotient}
		The restriction map $r^{T\times T}_{T} \colon \Kz^{T\times T}(G)\rightarrow\KTG$ is surjective with kernel $J(T)\cdot\Kz^{T\times T}(G)$ where $J(T)$ acts via $\alpha\cdot f \coloneqq p^{*}(\alpha)f$.
	\end{proposition}
	\begin{proof}
		Direct application of Corollary \ref{cor-subtor-quot-ktheory}. Given a $\mathbb{Z}$-basis $\{e_{1},\ldots,e_{n}\}$ of $X^{*}(T)$ we can choose $\chi_{i} = p^{*}(e_{i})$ for the application of Corollary \ref{cor-subtor-quot-ktheory}.
	\end{proof}

	\begin{proposition}
    \label{prop-torus-action-quotient-2}
		$\Kz^{T\times T}(G) \simeq \Kz^{B\times T}(G) \simeq \Kz^{T}(G/B)$.
	\end{proposition}
	\begin{proof}
		The first isomorphism follows from \ref{th-ignore-unip}. For the second isomorphism, there is an equivalence of categories $\Vect^{B \times T}(G) \simeq \Vect^{T}(G/B)$ as follows.
		
		\vspace{3mm}
		
		Consider the $B$-torsor $\pi \colon G \rightarrow B\backslash G$. Given a vector bundle $V \in \Vect^{T}(B\backslash G)$ there is an action of $B \times T$ on its pullback $\pi^*$ such that it becomes a $B \times T$-vector bundle on $G$. That is, pullback gives a functor $\pi^* \colon \Vect^T(B \backslash G) \rightarrow \Vect^{B\times T}(G)$. Invariant pushforward gives a quasi inverse $\Vect^{B\times T}(G) \rightarrow \Vect^{T}(B\backslash G)$. The $T$-equivariant isomorphism $i\colon G/B \rightarrow B \backslash G$ given by $gB \mapsto Bg$ induces an equivalence of categories $i^{*}\colon \Vect^{T}(B \backslash G) \simeq \Vect^{T}(G/B)$.
	\end{proof}

 \begin{proposition}
 \label{prop-torus-action-quotient-3}
		The map $\Kz^{B}(*) \simeq \Kz^{B\times G}(G) \rightarrow \Kz^{B\times T}(G)$ factors through an isomorphism $\Kz^{B}(*) \otimes_{\R(G)}\R(T) \simeq \Kz^{B\times T}(G)$. 
	\end{proposition}
	\begin{proof} Here we continue to write $\Kz^{B}(*)$ instead of the isomorphic $\R(B)$ (or $\R(T)$) in order to distinguish the different factors. 
 
    The first isomorphism follows from (\ref{prop-Thoma-subgp}). The restriction maps are compatible with the restriction map $\R(B \times G) \rightarrow \R(B \times T)$ so they factor through the tensor product with $\R(B \times T)$ over $\R(B\times G)$ (this is the same as the tensor product with $\R(T)$ over $\R(G)$).
			\[\begin{tikzcd}
			\Kz^{B\times T}(G) \arrow[r, "\sim"]& \Kz^{T}(G/B) \\ \Kz^{B\times G}(G)\otimes_{\R(G)}\R(T) \arrow[r, "\sim"] \arrow[u, "\alpha"] & \Kz^{G}(G/B)\otimes_{\R(G)}\R(T) \arrow[u, "\beta"]\\ \Kz^{B\times G}(G) \arrow[r, "\sim"] \arrow[r, "\sim"]\arrow[u]& \Kz^{G}(G/B) \arrow[u]
		\end{tikzcd}\]
	The map $\beta$ is an isomorphism by Proposition \ref{prop-equiv-kunneth-form} and so $\alpha$ must also be an isomorphism.
	\end{proof}

	\begin{proposition}
	\label{prop-torus-action-quotient-4}
		The ring homomorphism $$\theta \colon \Kz^{B}(*)\rightarrow \frac{\Kz^{B}(*)\otimes_{\R(G)}\R(T)}{J(T)(\Kz^{B}(*)\otimes_{\R(G)}\R(T))}\text{ given by }f\mapsto f\otimes 1$$ is surjective with kernel $I \Kz^{B}(*)$, where $I \coloneqq (c-\varphi(c) \mid c \in \R(G)) \subset \R(B)$.
	\end{proposition}
		
	\begin{proof}
		The element $1-\chi \in J(T) \subset R(T)$ acts on $\Kz^{B}(*)\otimes_{\R(G)}\R(T)$ as $(1-\chi)\cdot (g \otimes h) = (1\otimes 1 - \varphi(\chi)\otimes \chi^{-1})\cdot (g\otimes h) = g \otimes h - \varphi(\chi)g\otimes\chi^{-1}h$.
		
		\vspace{3mm}
		
		Surjectivity: It suffices to show that $1 \otimes \chi_{1}\ldots \chi_{n}$ lies in the image for all $\chi_{1},\ldots,\chi_{n} \in X^{*}(T)$. Now, 
		\begin{equation*}
			\begin{aligned}
				0 &= (1-\chi_{n})\cdot\ldots\cdot(1-\chi_{1}) \cdot(1\otimes \chi_{1}\ldots \chi_{n}) \\&= 1\otimes (\chi_{1}\ldots \chi_{n}) - \varphi(\chi_{1}\ldots \chi_{n})\otimes 1 \\&= 1\otimes (\chi_{1}\ldots \chi_{n}) - \theta(\varphi(\chi_{1}\ldots \chi_{n}))
			\end{aligned} 
		\end{equation*} so $1\otimes (\chi_{1}\ldots \chi_{n}) = \theta(\varphi(\chi_{1}\ldots \chi_{n}))$ is in the image.
		
		\vspace{3mm}
		
		Kernel: Firstly, note that for all $c \in \R(G)$ the element $(c-\varphi(c))\otimes 1 = 1\otimes c - \varphi(c)\otimes 1$ lies in $J(T)(\Kz^{B}(*)\otimes_{\R(G)}\R(T))$ so $I \subset \ker(\theta)$. Define a left inverse $\sigma$ for $\overline{\theta} \colon \frac{\Kz^{B}(*)}{I\Kz^{B}(*)}\rightarrow \frac{\Kz^{B}(*)\otimes_{\R(G)}\R(T)}{J(T)(\Kz^{B}(*)\otimes_{\R(G)}\R(T))}$ via $f \otimes g \mapsto \varphi(g)f$. This is well defined since $1\otimes c - c\otimes 1 \mapsto 0$ for all $c \in \R(G)$ and $(1-\chi)(f\otimes g) \mapsto 0$ for all $\chi \in X^{*}(T)$ and $f\otimes g \in \Kz^{B}(*)\otimes \R(T)$. It is a left inverse since $\sigma \circ \overline{\theta}(f) = f$ for all $f \in \Kz^{B}(*)$. Hence, $\overline{\theta}$ is injective and $\ker(\theta) = I\Kz^{B}(*)$.
		
	\end{proof}		
	\begin{corollary}
		\label{cor-kthy-tor-act}
		The map $r^{T\times T}_{T}\circ r^{T \times G}_{T\times T} \colon \R(T) \rightarrow \Kz^{T\times T}(G) \rightarrow \KTG$ is surjective with kernel $I\R(T)$ where $I \subset \R(G)$ is as in \ref{prop-torus-action-quotient-4}.
	\end{corollary}
 \begin{proof}
     Follows from \ref{prop-torus-action-quotient}-\ref{prop-torus-action-quotient-4}.
 \end{proof}

 \section{Computation of \texorpdfstring{$\Kz(\GZipmu)$}{}}
	\label{sec-main-th-proof}

	
	\begin{theorem}
		\label{th-main-intext}
		Let $G$ be a reductive group, $T \subset G$ a maximal torus, $L \subset G$ a Levi subgroup containing $T$ and $\varphi \colon G \rightarrow G$ be the Frobenius map. Suppose that $G^{\der}$ is simply connected. Consider the $L$-action on $G$ via $\varphi$-conjugation $l\cdot g = lg\varphi(l)^{-1}$. 
	 	$$\KLG = \R(L)/I\R(L)$$
		where $I \coloneqq \{c-\varphi(c) \mid c \in \R(G)\} \subset R(G)$.
	\end{theorem}

	\begin{proof}
	First, note that $G^{\der}$ simply connected implies that $L^{\der}$ simply connected (\ref{cor-simply-conn-Levi}) and hence $\R(L)\hookrightarrow \R(T)$ is faithfully flat. 	
	The following diagram of restriction maps is commutative. 
		\[\begin{tikzcd}
			& \R(T) \arrow[d, "\simeq"]& \R(L) \arrow[d, "\simeq"]\\ & \Kz^{T\times G}(G)\arrow[dl] \arrow[d, two heads, "r^{T\times G}_{T}"]& \Kz^{L\times G}(G) \arrow[d, "r^{L\times G}_{L}"] \arrow[l, hook']
			\\ \Kz^{T\times T}(G) \arrow[r, two heads] & \KTG & \KLG \arrow[l, hook']
		\end{tikzcd}\]
	By the equivariant Kunneth formula $r^{L\times G}_{L} \otimes _{\R(L)}\R(T) = r^{T\times G}_{T}$. Now we use that $\R(L)\hookrightarrow \R(T)$ is faithfully flat. Consider the sequence of $\R(L)$-modules
	\begin{equation}
 \label{tikz-diag-Kthy-L}
	    \begin{tikzcd}
		0 \arrow[r]& I\R(L) \arrow[r] & \R(L) \arrow[r, "r_{L}^{L\times G}"] & \KLG \arrow[r] & 0
	\end{tikzcd}
	\end{equation}
	Applying $-\otimes_{\R(L)}\R(T)$ gives the short exact sequence
	\[\begin{tikzcd}
		0 \arrow[r]& I\R(T) \arrow[r] & \R(T) \arrow[r] & \KTG \arrow[r] & 0 
	\end{tikzcd}\] using \ref{prop-equiv-kunneth-form} and the fact that $\R(T)$ is a flat $\R(L)$-module which tells us that $J\R(T) \simeq J \otimes _{\R(L)}\R(T)$ for the ideal $J := I\R(L) \subset \R(L)$.
	As $\R(L)\hookrightarrow \R(T)$ is faithfully flat, the sequence(\ref{tikz-diag-Kthy-L}) is exact and thus $r_{L}^{L\times G}$ is surjective with kernel $I\R(L)$.
	\end{proof} 
	\begin{corollary}
	\label{cor-main-thm}
	Let $(G,\mu)$ be a cocharacter datum and $L \coloneqq \Cent_{G}(\mu)$. Suppose that $G^{\der}$ is simply connected. 
	$$\Kz(\GZipmu) \cong \Kz^{L}(*)/I\Kz^{L}(*) \simeq \R(L)/I\R(L)$$
	\end{corollary}
	\begin{proof}
		The zip group $E$ sits in the split short exact sequence (\ref{tikzcd-ses-zipgroup}) so by Theorem \ref{th-ignore-unip} we have $\Kz(\GZipmu) \simeq \KLG$. 
	\end{proof}

 \section{Further methods using invariants}
	\label{sec-altproof-comp-invariants}
	A potential alternative method for deducing $\KLG$ from $\KTG$ is presented in this section. This relies on a description of $\KLG$ as a ring of invariants in $\KTG$. There is an action of the Weyl group $W$ on $\KTG$. However, in general this is not enough to cut out $\KLG$ as the ring of invariants (\ref{example-weyl-action}) and we must enrich this to an action by a larger ring (\ref{subsec-hecke-invariants}) on the module $K_{\bullet,\varphi}^{T}(G) = \oplus_{n}\KTnG$. If this action preserves the $\Kz$-piece then we can deduce the ring $\KLG$. 
	\subsection{Weyl group action on equivariant algebraic \texorpdfstring{$K$}{}-theory}
	
	Let $G$ be a reductive $k$-group and $X$ a $G$-variety. The action of the Weyl group $W = N_{G}(T)/T$ on $\Kz^{T}(X)$ is defined as follows. There is an action of $N_{G}(T)$ on the set of isomorphism classes $\Vect^{T}(X)/\simeq$ given by $n\cdot E \coloneqq n^{-1,*}E$. This action is trivial on $T \subset N_{G}(T)$ so it descends to an action of $W$ on $\Vect^{T}(X)/\simeq$ and hence on $\Kz^{T}(X)$.
	
	\vspace{3mm}
	
	Unfortunately, the analogue of Lemma \ref{le-weyl-action-repring} no longer holds. In general, it is \textbf{not} true that $ \Kz^G(X)\simeq \Kz^T(X)^W$, see (\ref{example-weyl-action}) for an example where this doesn't hold. This issue is resolved in \cite{MegumiLandweberSjamaar} by recourse to a Hecke action on $K_{\bullet}^{T}(X)$ which extends the Weyl group action, see below. For completeness we include the following counterexample to the analogue of Lemma \ref{le-weyl-action-repring}, adapted from \cite[4.8]{MegumiLandweberSjamaar}, \cite[4.5]{McLeod}.
	\begin{example}[$\Kz^{G}(X)\neq \Kz^{T}(X)^{W}$]
		\label{example-weyl-action}
		
		Consider $G = \SL_2$, $T = \Diag_{2}$. Then $\R(T) \cong \mathbb{Z}[x,x^{-1}]$ has a basis $\{1,x\}$ over the subring $\R(G) = \mathbb{Z}[x+x^{-1}]$. The Weyl group $W =\mathbb{Z}/2$ interchanges $x$ and $x^{-1}$. If $X$ is a $G$-variety then $\Kz^{T}(X) \simeq \R(T)\otimes_{\R(G)}\Kz^{G}(X) \cong \Kz^G(X)\oplus x\cdot\Kz^G(X)$. The Weyl invariants are
		\begin{equation*}
			\begin{aligned}
				\Kz^T(X)^W &= \{a+bx \mid a,b \in \Kz^G(X) \text{ and }2b = (x+x^{-1})b = 0\} \\&\cong
				\Kz^G(X)\oplus \{f \in \Kz^{G}(X) \mid 2f = f \text{ and }(x+x^{-1})f = f\}
			\end{aligned}
		\end{equation*}
		If $Y$ is a variety and $X = Y \times G$ is the trivial $G$-torsor over $Y$ then $\Kz^G(X) \simeq \Kz(Y)$ and $\R(G)$ acts by $(x+x^{-1}).c = 2c$. So for $Y$ such that $\Kz(Y)$ has $2$-torsion, the Weyl invariant ring $\Kz^T(X)^{W}$ is strictly larger than $\Kz^G(X)$. 
	\end{example}
	
	\subsection{Augmentation left ideal and Hecke invariants}
	\label{subsec-hecke-invariants}
	This section paraphrases the results of \cite{MegumiLandweberSjamaar} which build on work of Demazure. The aim is to enrich the Weyl group action to an action by a submodule of $\End_{\R(G)}(\R(T))$ called the \textit{augmentation left ideal}. For every root $\alpha \in \Phi(G,T)$ there is an operator $\delta_{\alpha} \in \End_{\R(G)}(\R(T))$ such that the collection satisfies the following property:
	\begin{lemma}[\cite{Demazure} Theorem 1]
		Let $w \in W$. Given two reduced expressions $w = s_{\alpha_{1}}\ldots s_{\alpha_{n}} = s_{\beta_{1}}\ldots s_{\beta_{n}}$ in terms of simple reflections, the composition $\delta_{\alpha_{1}}\ldots \delta_{\alpha_{n}} = \delta_{\beta_{1}}\ldots \delta_{\beta_{n}}$ and is denoted $\delta_{w}$.
	\end{lemma}
	Consider the subring $D \subset \End_{\R(G)}(\R(T))$ generated by the operators $\delta_{\alpha}$ and elements of $\R(T)$ acting by multiplication. The \textit{augmentation left ideal} is defined to be $\mathcal{I}(D) = \{\Delta \in D \mid \Delta(1)=0\}$. The augmentation ideal of $\mathbb{Z}[W]$ is contained in $\mathcal{I}(D)$ since $w(1) = 1 $ for $w \in W \subset \mathbb{Z}[W]$.
	\begin{proposition}[\cite{MegumiLandweberSjamaar} 6.5]
		$K_{\bullet}^{T}(X) = \oplus_{n}\Kn^{T}(X)$ is a module over the ring $D$.
	\end{proposition}
	Given a $D$-module $A$, define the Hecke invariants to be $$A^{\mathcal{I}(D)} \coloneqq \{f \in A \mid \Delta(f) = 0 \text{ for all }\Delta \in \mathcal{I}(D)\}$$ This is a submodule of $A$ over the ring $\R(G)$.
	\begin{theorem}[\cite{MegumiLandweberSjamaar}, 6.6]
		\label{th-hecke-action-inv}
		Suppose that $G^{\der}$ is simply connected and $G$ acts on a variety $X$. The restriction map $K_{\bullet}^{G}(X)\hookrightarrow K_{\bullet}^{T}(X)$ induces an isomorphism $K_{\bullet}^{T}(X)^{\mathcal{I}(D)} \simeq K_{\bullet}^{G}(X)$.
	\end{theorem}
	If the $D$-action preserves $\Kz^{T}(X)$ then this provides a method for deducing $\KLG$ from $\KTG$, which is not used in the main proof of the computation, but is more in line with Brokemper's proof in the Chow ring case \cite{Brokemper}. 
    \begin{proposition}
        Assume that $\mathcal{I}(D)$ preserves $\KTG$. Then $$\KLG \simeq \R(T)^{W_{L}}/I\R(T)^{W_{L}}$$
    \end{proposition} 
    \begin{proof}
		Our assumption means that 
    $\KLG \simeq \KTG^{\mathcal{I}(D_{L})}$. Given the description $\KTG = \R(T)/I\R(T)$, there are inclusions of rings $$\R(T)^{\mathcal{I}(D_{L})}/I\R(T)^{\mathcal{I}(D_{L})} \subset \KLG \subset \KTG$$ To show that the first inclusion is an equality it suffices to prove that $\KTG$ is a free module of rank $\rkLWeyl$ over both subrings \cite[1.4]{brokemper2016chow}. The equivariant Kunneth formula \ref{prop-equiv-kunneth-form} implies that $\KTG$ is a free module of rank $\rkLWeyl$ over $\KLG$ since $\R(T)$ is free of rank $\rkLWeyl$ over $\R(L)$. We have that $\R(T)^{\mathcal{I}(D_{L})} = \R(T)^{W_{L}}$ and $\R(T)/I\R(T)$ is a free module of rank $\rkLWeyl$ over $\R(T)^{W_{L}}/I\R(T)^{W_{L}}$ since $\R(T)^{W_{L}}=\R(L)\hookrightarrow \R(T)$ is faithfully flat.
		
		\vspace{3mm}
		
		 Hence, $\KLG \cong \R(T)^{W_{L}}/I\R(T)^{W_{L}} \cong \R(L)/I\R(L)$.
	\end{proof}
    
\bibliographystyle{amsplain}
\bibliography{references}

\end{document}